\documentclass[preprint]{elsarticle}
\usepackage{amsmath, amsthm, amssymb, graphics, graphicx}
\usepackage{subfigure}
\usepackage{latexsym}
\usepackage{graphics,epsfig}
\usepackage{amsfonts}
\usepackage[english]{babel}
\usepackage{verbatim}
\usepackage[ruled,vlined,linesnumbered]{algorithm2e}
\usepackage{algorithmic}

\usepackage{pstricks}
\usepackage{pst-node,pst-tree,pst-eps,xspace}
\IfFileExists{pst-beta.tex}
            {\input{pst-beta.tex}}
             {\RequirePackage{pst-tree}}

\newcommand{\mathsym}[1]{{}}
\newtheorem{theorem}{Theorem}
\newtheorem{lemma}{Lemma}
\newtheorem{proposition}[lemma]{Proposition}
\newtheorem{corollary}[lemma]{Corollary}
\newtheorem{definition}{Definition}
\newtheorem{conjecture}{Conjecture}

\def \seq #1 #2 {#1_1#1_2\dots #1_{#2}}
\def \tauu #1 {$\tau_1,\tau_2\dots,\tau_{#1}$}

\newcommand{\s}{S}

\begin{document}
\begin{frontmatter}
\title{Combinatorial bijections from hatted avoiding permutations in $S_n(132)$ to generalized Dyck and Motzkin paths}
\author[fn1]{Phan Thuan DO}
\address[fn1]{Department of Computer Science, Hanoi University of Science and Technology, 04 Dai Co Viet road, Hanoi, Vietnam}
\ead{thuandp@soict.hut.edu.vn}
\author[fn2]{Dominique ROSSIN}
\ead{rossin@lix.polytechnique.fr}
\address[fn2]{Laboratoire d'Informatique, $\acute{E}$cole Polytechnique, Route de Saclay, 91128 PALAISEAU Cedex, France}
\author[fn3]{Thi Thu Huong TRAN}
\ead{ttthuong@math.ac.vn}
\address[fn3]{Institute of Mathematics, 18 Hoang Quoc Viet, Hanoi, Vietnam}
\date{}
\begin{abstract} We introduce a new concept of permutation avoidance pattern called hatted pattern, which is a natural generalization of the barred pattern. We show the growth rate of the class of permutations avoiding a hatted pattern in comparison to barred pattern. We prove that Dyck paths with no peak at height $p$, Dyck paths with no $ud\dots du$ and Motzkin paths are counted by hatted pattern avoiding permutations in $\s_n(132)$ by showing explicit bijections. As a result, a new direct bijection between Motzkin paths and permutations in $\s_n(132)$ without two consecutive adjacent numbers is given. These permutations are also represented on the Motzkin generating tree based on the Enumerative Combinatorial Object (ECO) method.  
\end{abstract}
\end{frontmatter}

\section{Introduction} \label{S:1}
Studying restricted permutations bijectively related to known combinatorial objects has always received great attention \cite{BDPP00,BD06,Cal07,CK08,Kra01,Stu09,Wes95}. A permutation $\pi$ of length $n$ is called avoiding a permutation $\tau$ of length $k$, called a pattern $\tau$, if $\pi$ does not contain any subsequence of length $k$ order-isomorphic to $\tau$. Two of the essential studies relevant to restriction permutations are to enumerate permutations avoiding all patterns in a given set, and to find a set $T$ of permutation patterns such that permutations avoiding $T$ count known combinatorial sequences as Catalan, Motzkin, Schr{\"o}der, Fibonacci and many others in Sloane \cite{Sloane}. The first one has been done partly for some specific sets of patterns where almost patterns are of length $\leq 4$. A detailed statistic of enumerated patterns is given by Elizalde \cite{Eli04}. Marcus and Tardos in \cite{MT04} proved the Wilf-Stanley's conjecture, which says that the number of permutations of length $n$ avoiding a pattern is bounded by a power function of $n$ which is much asymptotically smaller than the number of permutations. For the second one, many types of generalized permutation patterns, such as barred, dashed, dotted patterns are introduced \cite{BS00,Bar11,Pud10}. We are specially interested in the barred pattern introduced in an exposition by Pudwell \cite{Pud10}. In our point of view, the barred pattern is much meaningful because a numerous essential sequences are counted by permutations avoiding a mixture of original patterns and barred patterns. Furthermore, although Wilf-Stanley's conjecture does not hold for the barred pattern, we still expect to generalize it into other one such that the number of permutations restricted on it grows either fast enough as the number of unrestricted permutations or slow enough as the number of permutations avoiding an original pattern. 

In combinatorics, showing bijective proofs between two finite sets with same cardinality has been specially got great interests beside proofs using generating functions or recursive formulas. In this context, a list of combinatorial bijections between Dyck paths of length $2n$ and permutations of length $n$ avoiding a pattern of length $3$ is presented in \cite{CK08}. Many of them are induced as compositions of different bijections where one of their  components is often the standard bijection, which is between Dyck paths and permutations avoiding $132$, given by Knuth \cite{Knu69,Knu73}.  

In this paper, we introduce a new type of permutation pattern, called hatted pattern, which is a natural generalization of the barred pattern where the avoiding condition is made weaker. We show all cases in which the barred pattern coincides to the hatted pattern and prove that Wilf-Stanley's conjecture does not hold for a class of hatted patterns (Section \ref{S:2}). By developing various versions of the standard bijection, we point out some hatted patterns and prove that Dyck paths with no peaks at a given height, Dyck paths without $udd\dots du$ are bijective to permutations avoiding $132$ and one of these hatted patterns  (Section \ref{S:3}). Furthermore, Callan \cite{Cal04} provides a bijection from Dyck paths of length $(2n+2)$ with no $udu$ to Motzkin of length $n$. To the best of our knowledge although there are a lot of bijections to restricted permutations from Dyck paths \cite{Cal07,CK08,Kra01,Stu09}, there has not been yet any one from Motzkin paths. This suggests us to construct a direct bijection between them. We present such bijection from Motzkin paths of length $n$ to $(n+1)$-length permutations avoiding $132$ and without any appearance of two adjacent consecutive integers in Section~\ref{S:4}. Finally, based on the ECO method, we show that the Motzkin generating tree given in \cite{BDPP00} is also coded by these permutations. 

\section{Hatted pattern}\label{S:2}
In this section, we introduce a new type of pattern, called \emph{hatted pattern} after giving definitions of permutation pattern, barred pattern. We notice that the hatted pattern is a natural generalization of the barred pattern and show all patterns in which the hatted pattern is properly different to the barred pattern. We show that while the number of permutations avoiding one of those barred pattern grows in a exponential function, the number of permutations avoiding one of those hatted pattern grows in factorial function. 

Let $q=q_1\dots q_m$ be a string of numbers. The \emph{reduction} of $q$, denoted by $red(q)$, is the string obtained from $q$ by replacing $i$th smallest element of $q$ by $i$. For example $red(357136)=235124$ and if $q$ is a string of different numbers then $red(q)$ is a permutation. We denote by $\s_n$ the set of permutations on $[n] = \{1,2,\dots,n\}$. Permutations in this paper are written in one-line notation. Let $\tau=\tau_1\dots\tau_k\in\s_k$. A permutation $\pi=\pi_1\pi_2\dots\pi_n \in \s_n$ is called \emph{containing} the pattern $\tau$ if there exists a sequence of indices $1\leq i_1<i_2<\cdots<i_k\leq n$ such that $red(\pi_{i_1}\pi_{i_2}\dots\pi_{i_k})=\tau$. Otherwise, $\pi$ avoids $\tau$ or $\pi$ is $\tau$-avoiding.  

Denote by $\tau_{(i)}$ the permutation which is identical to $\tau$ with the element $\tau_i$ is marked. We say that \emph{$\pi$ avoids $\tau_{(i)}$ by barred type}, or simply \emph{$\pi$ avoids $\bar\tau_{(i)}$} where the element $\tau_{i}$ of $\tau$ is now marked by a bar, if every subsequence of $(k-1)$ elements $\pi_{j_1}\pi_{j_2}\dots\pi_{ j_{k-1}}$ satisfying  $$red(\pi_{j_1}\pi_{j_{2}}\dots\pi_{j_{k-1}})=red(\tau\backslash \tau_i),$$ can be expanded into a subsequence $\pi_{j_1}\cdots\pi_{j_{i-1}}\pi_{j_*}\pi_{j_{i}}\cdots \pi_{j_k}$ such that  $$red(\pi_{j_1}\cdots\pi_{j_{i-1}}\pi_{j_*}\pi_{j_{i}}\cdots \pi_{j_k})=\tau.$$

Furthermore, $\pi$ is called \emph{avoiding $\tau_{(i)}$ by hatted type}, or simply \emph{$\pi$ avoids $\hat\tau_{(i)}$} where the element $\tau_i$ is now marked by a hat, if every subsequence of $(k-1)$ elements $\pi_{j_1}\pi_{j_2}\dots\pi_{ j_{k-1}}$ satisfying  $$red(\pi_{j_1}\dots\pi_{j_{k-1}})=red(\tau\backslash \tau_i),$$ 
can be expanded into a subsequence of $k$ elements $\pi_{t_1}\pi_{t_{2}}\dots \pi_{t_k}$ of $\pi$ such that $\{j_1,j_2,\dots, j_{k-1}\} \subseteq \{t_1,t_2,\dots, t_{k}\}$ and $$red(\pi_{t_1}\pi_{t_2}\dots \pi_{t_{k}})=\tau.$$ 

Given a set of patterns $T$ (may contain barred and hatted patterns), we are interested in enumerating $S_n(T)$, the set of permutations of length $n$ avoiding all patterns in $T$. 

\noindent 
{\bf Example:} 
$$\s_3(21)=\{123\}, \s_3(\bar 21)=\emptyset, \s_3(\hat 21)=\{321, 312, 231\}.$$

\noindent{\bf Remark:} 
\begin{enumerate}
\item If $\pi$ avoids $\hat\tau$, then $\pi$ avoids $\tau$ for all subsequences of $\pi$ avoiding $\tau\backslash \tau_i$;
\item The added element in each expanded subsequence in the definition of the barred pattern avoiding permutation plays exactly the role as the barred element in the pattern, whereas this may be more active in the one of hatted pattern avoiding permutation (see Figure \ref{F:pos_hat}). Hence, $S_n(\bar\tau_{(i)})\subseteq S_n(\hat\tau_{(i)})$ and the reverse inclusion is not true in general. For example, $2143\in \s_4(2\hat 13)$ and $2143\notin \s_4(2\bar 13)$ since the increasing subsequence $13$ cannot be extended into pattern $213$ by barred type. Later, we will give a characterization for patterns where this equality holds. 
\end{enumerate}
\begin{figure}[h]
\centering
\subfigure[into pattern $213$ by barred type] {\label{F:pos1}\includegraphics[height=3.5cm, width=5cm]{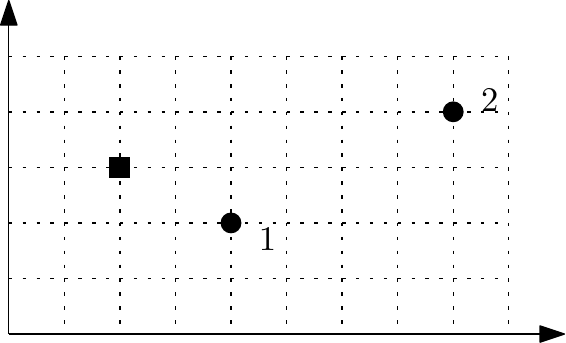}}\hspace{1cm}
\subfigure[into pattern $213$ by hatted type] {\label{F:pos2}\includegraphics[height=3.5cm, width=5cm]{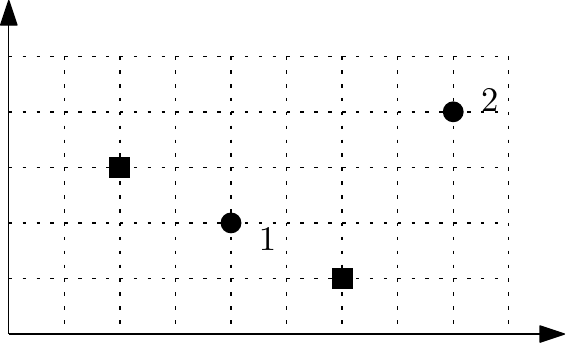}}
\caption{Possible positions to expand a pattern $12$}\label{F:pos_hat}
\end{figure}


The following lemma is implied from the definition. 
\begin{lemma}\label{L:hat1}
Let $\tau\in \s_k$ and $1\leq i,j\leq k$ such that $red(\tau\backslash \tau_i)=red(\tau\backslash \tau_j)$. Then  $\s_n(\hat\tau_{(i)})=\s_n(\hat\tau_{(j)})$.
\end{lemma}
Now, let $\chi$ be one of the three trivial bijections on permutations (reverse, complement, inverse). The next lemma is straightforward from the definition. 

\begin{lemma}\label{L:hat2}
 Let $\tau,\tau^\prime\in \s_k$ and $1\leq i,j\leq k$. If the map $\chi$ satisfies
\begin{itemize}
	\item $\chi(\tau)= \tau^\prime,$
	\item $\chi(red(\tau\backslash \tau_i))= red(\tau^\prime\backslash\tau^\prime_j),$
\end{itemize}
then  
$\chi: \s_n(\hat\tau_{(i)}) \to \s_n(\hat\tau^\prime_{(j)})$ 
is a bijection. Consequently, $|\s_n(\hat\tau_{(i)})| =|\s_n(\hat\tau^\prime_{(j)})|$. 
\end{lemma}

For instance, we have 
\begin{itemize}
	\item $|\s_n(143\hat2)|=|\s_n(\hat2341)|=|\s_n(2\hat341)|=|\s_n(23\hat41)|$ (by Lemma \ref{L:hat2} where $\chi$ is the reverse and Lemma \ref{L:hat1});
	\item $|\s_n(143\hat2)|=|\s_n(412\hat3)|=|\s_n(41\hat23)|=|\s_n(4\hat123)|$ (by Lemma \ref{L:hat2} where $\chi$ is the complement and Lemma \ref{L:hat1});
	\item $|\s_n(4\hat123)|=|\s_n(\hat2341)|=|\s_n(2\hat341)|=|\s_n(23\hat41)|$ (by Lemma \ref{L:hat2} where $\chi$ is the inverse and Lemma \ref{L:hat1}).
\end{itemize}


To finish this section, we show the exact relation between hatted pattern and barred pattern
\begin{proposition}\label{P:equ} $S_n(\hat\tau_{(i)})=S_n(\bar\tau_{(i)})$ for all $n=1,2,\dots$ if and only if $\tau_i$ is different from $\tau_{i-1}-1, \tau_{i-1}+1, \tau_{i+1}-1, \tau_{i+1}+1$. 
\end{proposition}
\begin{proof}
If $\tau_{i-1}=\tau_{i}-1$ or $\tau_{i-1}=\tau_i+1$, then $\tau_1\dots\tau_i\tau_{i-1}\tau_{i+1}\dots\tau_k$ avoids $\hat\tau_{(i)}$ but not avoids $\bar\tau_{(i)}$ and so $S_k(\hat\tau_{(i)})\supsetneq S_k(\bar\tau_{(i)})$. Argument similarly for $\tau_{i+1}$ and the inference direction is proved.  

Conversely, we prove that for each $\pi\in S_n(\hat\tau_{(i)})$ then $\pi$ also avoids $\bar\tau_{(i)}$. Take $A=\pi_{t_1}\pi_{t_2}\dots \pi_{t_{k-1}}$ being an arbitrary subsequence of $\pi$ such that $red(A)=red(\tau\backslash\tau_i)$. Assume that $A$ is expanded into a subsequence $B$ such that $red(B)=\tau$ by adding one another element of $\pi$ before $\pi_{t_{i-1}}$. So the element $\pi_{t_{i-1}}$ which plays the role as the element $\tau_{i-1}$ in $A$ now will play the role as the element $\tau_i$ in $B$. Hence, $\pi_{t_{i-1}}$ is greater than $\tau_i-1$ elements in $B$ and so it is greater than $\tau_i-2$ elements in $A$. Taking the reduction we get $\tau_{i-1}>\tau_i-2$. Furthermore, $\pi_{t_{i-1}}$ can not be the element greater than $(\tau_{i}+1)$ other elements in $B$ and so it can not the element greater than $\tau_{i}$ other elements in $A$. Taking the reduction we get $\tau_{i-1}<\tau_i+2$. Hence, $\tau_{i-1}$ is equal to either $\tau_i-1$ or $\tau_{i}+1$ which is a contradiction to the hypothesis. Similarly, $A$ can not be expanded into a subsequence of reduction $\tau$ by adding one another element of $\pi$ after $\pi_{t_{i}}$ , otherwise $\tau_{i+1}$ is equal to either $\tau_i-1$ or $\tau_i+1$. Therefore, $A$ is only expanded into pattern $\tau$ by barred type  and hence $\pi$ avoids $\bar\tau_{(i)}$. 
\end{proof}
\begin{proposition}\label{P:fac}
Let $\tau\in \s_k$ and $1\leq i\leq k$ such that $\tau_i$ is equal to one of four values: $\tau_{i-1}-1, \tau_{i-1}+1, \tau_{i+1}-1, \tau_{i+1}+1$, \emph{i.e.} $S_n(\bar\tau_{(i)})\subsetneq S_n(\hat\tau_{(i)})$. Then the growth rate of $|S_n(\hat\tau_{(i)})|$ is factorial. 
\end{proposition}
\begin{proof}

By Proposition \ref{P:equ}, it is sufficient to prove for the case $\tau_i=\tau_{i+1}+1$. We construct a mapping $f:\s_n\to \s_{nk}$ as follows. Each $\pi=\pi_1\pi_2\cdots \pi_n\in \s_n$ will map to $f(\pi)$ which is obtained from $\pi$ by replacing each element $\pi_i$ of $\pi$ with an increasing sequence of $k$ consecutive integers $\pi_ik-k+1,\pi_ik-k+2,\dots,\pi_ik$. For instance, $f(312)=789123456$. It is clear that $f$ is injective. We prove that $f(\pi)$ avoids $\hat\tau_{(i)}$ for all $\pi\in \s_n$. So that $|\s_{nk}(\hat\tau_{(i)})|\geq |\s_n|$ and $|\s_{n}(\hat\tau_{(i)})|\geq \left[\frac{n}{k}\right]!$ which is of factorial form as we 
desire. Thus, take a subsequence of $f(\pi)$ which has reduction $red(\tau\backslash \tau_i)$. Then this subsequence has $k-1$ elements and the element playing the role of $\tau_{i+1}$ in it will be in a sequence of $k$ adjacent consecutive integers in $f(\pi)$. By the pigeonhole principle, there is at least 1 element in this sequence of $k$ adjacent consecutive integers not in the subsequence of reduction $red(\tau\backslash \tau_i)$ above. We choose among them the element nearest (in $f(\pi)$) to the element playing the role of $\tau_{i+1}$ above. The subsequence inserted this element forms a subsequence of pattern $\tau$ in $f(\pi)$. So that $f(\pi)$ avoids $\hat\tau_{(i)}$.  
\end{proof}
Notice that Lemma 2 in \cite{Pud10} saying that the cases where barred pattern is proper contained in hatted pattern, the barred pattern becomes normal pattern and so that Wilf-Stanley's conjecture holds for barred pattern but by Propositions \ref{P:equ}, \ref{P:fac} it fails for hatted pattern. It suggests us to make the following conjecture
\begin{conjecture} The growth rate of $|S_n(\hat\tau_{(i)})|$ is factorial for all hatted pattern $\hat\tau_{(i)}$. 
\end{conjecture}
\section{Hatted patterns visiting Dyck paths}\label{S:3}
In this section, we study two subclasses of $\s_n(132)$: $\s_n(132,(p-1)(p-2)\dots 2\hat 1p)$ and $\s_n(132,12\dots \hat p)$. We prove in Section \ref{S:32} that the standard bijection \cite{Knu69} between $S_n(132)$ and Dyck $n$-paths restricted on $\s_n(132,12\dots \hat p)$ is bijective to Dyck paths with no peaks at height $p$. This proof uses the non-recursive description version of the standard bijection given by Krattenthaler \cite{Kra01}. Furthermore, in Section \ref{S:31}, we propose a modification of the standard bijection and prove that its restriction on $\s_n(132,(p-1)(p-2)\dots 21p)$ is bijective to Dyck $n$-paths with no $udd\dots d u$. Last we present the Wilf equivalence for these two classes of permutations by Simion-Schmidt's bijection. 
\subsection{Dyck paths with no peaks at height $p$.} \label{S:32} 
We first recall some preliminary definitions 
\begin{definition} Let $n$ be a positive integer. 
\begin{itemize}
\item[(i)] A \emph{Dyck $n$-path} is a lattice path in the integer plane starting at $(0,0)$ and ending at $(0,2n)$ which consists of $n$ \emph{up-steps} $(1,1)$, $n$ \emph{down-steps} $(1,-1)$ and never runs bellow $x$-axis. 
\item[(ii)] A \emph{peak} of a Dyck path is the point created by an up-step followed by a down-step. The height of the peak is the $y$-coordinate of the point.
\item[(iii)] A \emph{valley} of a Dyck path is the point created by a down-step followed by an up-step.
\end{itemize}
\end{definition}
Each Dyck $n$-path is also represented by a word of length $2n$ on the alphabet $\{u,d\}$ , where $u$ and $d$ substitute for up and down step respectively. For example, the Dyck $6$-path in Figure \ref{F:dyck4} is represented by $uuududduududdd$. 


A permutation is represented into left-to-right minimal blocks. For abbreviation, we write $LTR$ (resp. $LTRM$) instead of left-to-right (resp. left-to-right minimal). Let $\pi= \seq{\pi} n \in \s_n$. An element $\pi_i$ of $\pi$ is called a $LTR$ minimum of $\pi$ if $\pi_i<\pi_j$ for all $j<i$ and in this case $i$ is called a $LTRM$ index of $\pi$. Assume that $1=i_1<i_2<\cdots<i_k$ are all  $LTRM$ indices of $\pi$. Then $\pi$ is uniquely represented  into $k$ blocks. Each one is called a $LTRM$ block starting by an LTRM index. $\pi$ can be separated by parentheses as follows: 
$$\pi=(\pi_{i_1}\dots\pi_{i_2-1})(\pi_{i_2}\dots\pi_{i_3-1})\cdots(\pi_{i_k}\dots\pi_n).$$
For instance, the elements $5,4,2,1$ are $LTR$ minima of $\pi=(5)(46)(2)(137)$ with $LTRM$-blocks $\{5,46,2,137\}$.

Particularly, when $\pi$ is $132$-avoiding, we have the following lemma:
\begin{lemma}[\cite{Bon06}]\label{L:ltrm}
Let $\pi\in \s_n(132)$ with $LTRM$ indices $1=i_1<i_2<\cdots<i_k$. Then, the $LTRM$-blocks of $\pi$ satisfy the following conditions
\begin{itemize}
\item[i)] the first elements of the blocks are decreasing from left to right.
\item[ii)] each block is an increasing sequence.
\end{itemize}
\end{lemma}
We now present the non-recursive version of the standard bijection between $S_n(132)$ and Dyck $n$-paths given by Krattenthaler \cite{Kra01}. Denote this bijection by $\phi$. The recursive version will be mentioned in Section \ref{S:31}. Let $\pi\in \s_n(132)$. To generate a Dyck path from $\pi$ we read each element of $\pi$ from left to right as follows.
The path starts from $(0,0)$. When $\pi_i$ is read, 
\begin{itemize} 
\item we add up-steps to the path up-steps until it hits the height $h_i+1$,  where $h_i =|\{\pi_k: \pi_k>\pi_i \text{ and } k>i\}|$ is the number of elements after $\pi_i$ in  $\pi$ and greater than $\pi_i$;
\item we then add one down-step to the path. 
\end{itemize}
\begin{figure}[h]
\centering
\includegraphics[width=8cm]{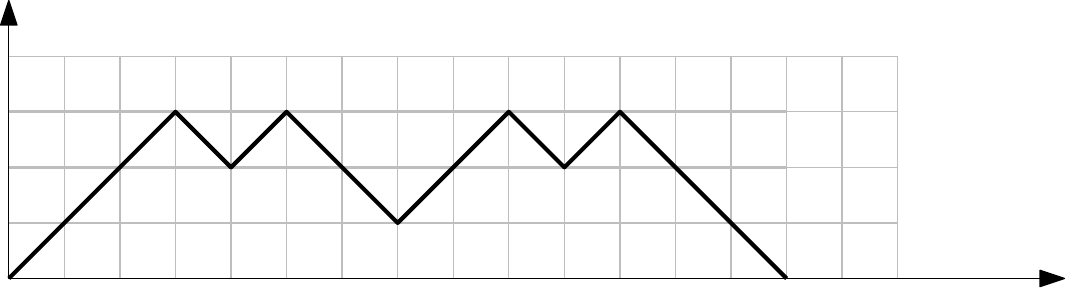}
\caption{The Dyck path with no peaks at height $2$ which is the image of $5462137\in\s_7(132, \hat 123)$ by $\phi$}
\label{F:dyck4}
\end{figure}

For example, $\phi$ will map $\pi=5462137$ to the Dyck $7$-path showed in Figure \ref{F:dyck4}. It is noticeable that $\pi\in \s_n(132, \hat 123)$ and $\phi(\pi)$ has no peaks at height $2$. We have the followings:
\begin{lemma}\label{L:sta}
Let $\pi=\pi_1\dots\pi_n\in \s_n(132)$ and $D_n=\phi(\pi)$. Then, reading each $LTR$ minimum in $\pi$ corresponds to a peak in $D_n$ by the above construction. Furthermore, the height of the peak in $D_n$ corresponding to reading $\pi_i$ of $\pi$ is $h_i+1$.
\end{lemma}
\begin{proof}
By Lemma \ref{L:ltrm}, $\pi_i$ is a $LTR$ minimum if and only if $\pi_i<\pi_{i-1}$. Hence, $\pi_i$ is a $LTR$ minimum if and only if $h_i\geq h_{i-1}$. Therefore, when $\pi_i$ is read, the path must go up at least one step from height $h_{i-1}$ before going down one step from $h_i+1$ to $h_i$, this creates a peak in $D_n$. By the description of $\phi$, the height of this peak is $h_i+1$.

\end{proof}
\begin{theorem}\label{T:cat}
Let $n,p$ be positive integers such that $1\leq p\leq n$. The map $\phi$ restricted on $\s_n(132, \hat 12\dots {(p+1)})$ is bijective to Dyck $n$-paths with no peak at height $p$.
\end{theorem}
\begin{proof}
Let $\pi\in\s_n(132, \hat 12\dots (p+1))$, and let $D_n=\phi(\pi)$. On the contrary, supposing that $D_n$ has a peak at height $p$. By Lemma \ref{L:sta} this peak corresponds to $\pi_{i_0}$ satisfying
\begin{itemize}
\item[(i)] $\pi_{i_0}$ is a LTR minimum of $\pi$. 
\item[(ii)] there are exactly $(p-1)$ elements of $\pi$ after $\pi_{i_0}$ and greater than $\pi_{i_0}$, says $\pi_{i_1}, \dots, \pi_{i_{p-1}}$ with $i_0<i_1<\cdots<i_{p-1}.$
\end{itemize} 
Since $\pi$ avoids $132$, the subsequence $\pi_{i_0}\dots \pi_{p-1}$ is increasing, otherwise there exists $t$, where $0<t<p-1$, such that $\pi_{i_t}>\pi_{i_{t+1}}$ and $red(\pi_{i_0} \pi_{i_t} \pi_{i_{t+1}}) = 132$. So $red(\pi_{i_0}\pi_{i_1}\dots\pi_{i_{p-1}})=12\dots p$. Furthermore, we can not expand this subsequence into the one whose reduction is $12\dots p(p+1)$ by inserting any element of $\pi$ after $\pi_{i_0}$ (by (ii)) nor before $\pi_{i_0}$ (by (i)). Hence, $\pi$ contains the pattern $\hat 12\dots p(p+1)$ which is a contradiction.   

Conversely, let $D_n$ be a Dyck $n$-path with no peak at height $p$ and let $\phi^{-1}(D_n)=\pi_1\dots\pi_n$. Then $\phi^{-1}(D_n)$ avoids $132$. We prove that $\phi^{-1}(D_n)$ avoids $\hat 12\dots(p+1)$. On the contrary, there exists a subsequence  $\pi_{i_0}\pi_{i_1}\cdots\pi_{i_{p-1}}$ of  $\phi^{-1}(D_n)$, whose reduction is $12\dots p$, which can not be expanded into the pattern $12\dots (p+1)$ in $\phi^{-1}(D_n)$. Then $\pi_{i_0}$ is a $LTR$ minimum (otherwise the element smaller than $\pi_{i_0}$ and before $\pi_{i_0}$ in $\pi$ can be inserted to the subsequence and forms a pattern $12\dots (p+1)$) and there are exactly $(p-1)$ elements of $\phi^{-1}(D_n)$ after $\pi_{i_0}$ and greater than $\pi_{i_0}$ (otherwise the subsequence either can be expanded into pattern $12\dots (p+1)$ or contains pattern $132$). By Lemma \ref{L:sta}, reading $\pi_{i_0}$ creates a peak at height $p$ in $D_n$ which is a contradiction. 
\end{proof}
The Corollaries \ref{C:Fin}, \ref{C:Cat} bellow are immediate from Theorem \ref{T:cat} and the following Propositions \ref{P:Dyc.h1}, \ref{P:Dyc.h2}.
\begin{proposition}[\cite{Deu99}]\label{P:Dyc.h1}
Dyck $n$-paths with no peaks at height $1$ counts the $n$-th Fine number.
\end{proposition}
\begin{proposition}[\cite{PW01}]\label{P:Dyc.h2}
Dyck $n$-paths with no peaks at height $2$ counts the $(n-1)$-th Catalan number.
\end{proposition}
\begin{corollary}\label{C:Fin} $\s_n(132,\hat 12)$ counts the $n$th Fine number.
\end{corollary}
\begin{corollary}\label{C:Cat} $\s_n(132,\hat 123)$ counts the $(n-1)$-th Catalan number.
\end{corollary}

In this case, it is interesting that $|\s_n(132,\hat 123)|=|S_{n-1}(132)|$.

\begin{corollary} 
When $p>n$, $\s_n(132,\hat 123\ldots (p+1))=S_{n}(132)$.
\end{corollary}
This shows a discrete continuity from the Catalan sequence to itself.
\subsection{Dyck paths with no $udd\cdots du$}\label{S:31}
Let $A$ be a sequence of distinct integers, then $A=A^L m A^R$, where $m$ is the greatest element in $A$ and $A^L$, $A^R$ are subsequences of $A$. We define recursively a map, denoted by $\theta$, on the set of integer sequences as follows:
\begin{itemize}
\item $\theta(\epsilon)=\epsilon$;
\item $\theta(A)=\theta(A^L)u_{m}\theta(A^R)d_{m}$. 
\end{itemize}
This recursive process will finally give a Dyck path with each step indexed by an integer in $A$.  We call $\theta(A)$ an \emph{indexed} Dyck path. 
\begin{lemma}\label{L:bij} The map $\theta$ restricted on $\s_n(132)$ is bijective to Dyck $n$-paths.
\end{lemma}
\begin{proof} 
Let $\pi_1,\pi_2\in S_n(132)$ such that $ \theta(\pi_1)$ and $\theta(\pi_2)$ give the same Dyck $n$-path without indexing $D_n$. By the definition of $\theta$, $D_n$ is illustrated as in Figure \ref{F:theta}. So its last $d$ is indexed as $d_n$ and the last $u$ which starts from a point on the $x$-axis is indexed as $u_n$. Therefore, $red(\pi_1^L)=red(\pi_2^L)$ and $red(\pi_1^R)=red(\pi_2^R)$. On the other hand, since $\pi_1,\pi_2\in S_n(132)$, their elements are distributed as in Figure \ref{F:132}, \emph{i.e.} all elements of the left part are greater than those of the right part. Therefore, $\pi_1^R=\pi_2^R$ and recursively, $\pi_1^L=\pi_2^L$. Hence, $\theta$ is injective. The surjectivity of $\theta$ is deduced clearly from the recursive definition of $\theta$.
\begin{figure}[h]
\centering
\subfigure[Dyck path given recursively by $\theta$]{\label{F:theta}\includegraphics[width=5cm]{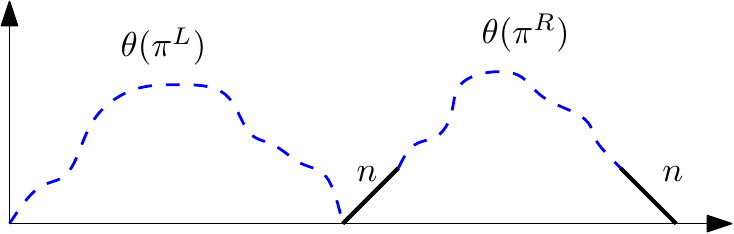}}\hspace{1cm}
\subfigure[$132$-avoiding permutation]{\label{F:132} \includegraphics[width=4.5cm]{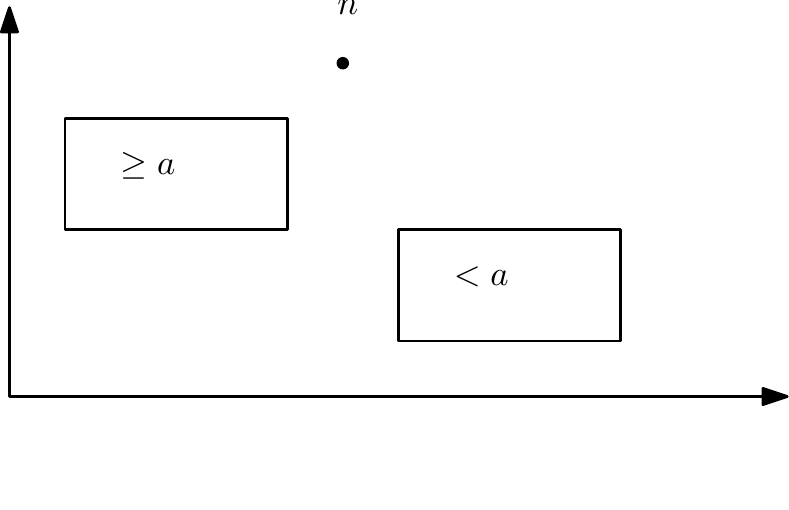}}
\caption{}
\end{figure}
\end{proof}

\noindent{\bf Remark:} 
\begin{itemize}
\item $\theta$ is a modified generalization of the following standard bijection \cite{CK08}:
$$\phi(\pi)=u\phi(\pi^L) d \phi(\pi^R).$$
The non-recursive version of this standard bijection has been introduced in Section \ref{S:32}. 
\item Although the recursive formulas of $\phi$ and $\theta$ are quite similar, their given Dyck paths are completely different. For instance, see  Figure \ref{F:dyck},
$$\phi(452361)=u\theta(4523)d\theta(1)=uu\phi(4)d\phi(23)dud=uuuddu\phi(2)ddud=uuudduudddud,$$ 
$$\theta(452361)=\theta(4523)u_6\theta(1)d_6=\theta(4)u_5\theta(23)d_5u_6u_1d_1d_6=u_4d_4u_5u_2d_2u_3d_3d_5u_6u_1d_1d_6.$$ 
\item The standard bijection $\phi$ restricted on $\s_n(132,(p-1)\dots \hat 1p)$ is not bijective to Dyck $n$-paths with no $ud\cdots du$. For instance, $5462137\in \s_7(132,2\hat 13)$ but $\phi(5462137)$ contains $udu$ (see Figure \ref{F:dyck4}). 
\end{itemize}


We call each pair $(u_k, d_k)$ in an indexed Dyck path $\theta(\pi)$ a \emph{well-matching} pair. Conversely, by Lemma \ref{L:bij}, given a Dyck $n$-path, its steps are uniquely indexed by integers on $[n]$ to be an image by $\theta$. The following properties are straightforward from the construction of $\theta$:
\begin{enumerate} 
\item[(i)] If two well-matching pairs are overlapped then the one with smaller index is nested within the other, {\it e.g.} ($\cdots u_5 \cdots u_2 \cdots d_2 \cdots d_5 \cdots$);
\item[(ii)] Two consecutive steps $ud$ creating a peak are indexed $u_kd_k$ for some $k$;
\item[(iii)] Two consecutive steps $du$ creating a valley are indexed $d_{k-1}u_k$. 
\end{enumerate}

\begin{figure}[h]
\centering
\subfigure[by the standard bijection $\phi$] {\label{F:dyck1}\includegraphics[width=5cm]{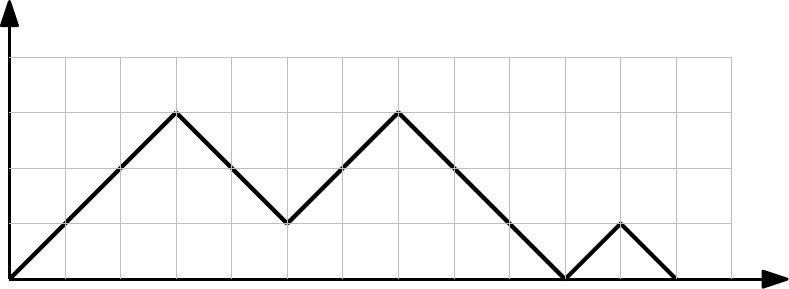}}\hspace{1cm}
\subfigure[by the modified bijection $\theta$]{\label{F:dyck2}\includegraphics[width=6cm]{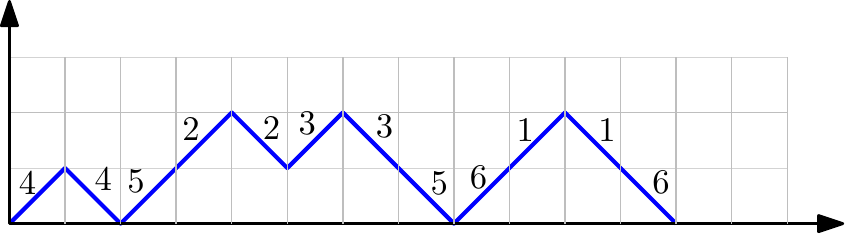}}
\caption{The images of $452361\in S_6$}\label{F:dyck}
\end{figure}

\begin{theorem}\label{T:udu}
Let $n,p$ be positive integers and $p\geq 3$. The map $\theta$ restricted on $\s_n(132,(p-1)(p-2)\cdots 2\hat 1p)$  is bijective to Dyck $n$-paths with no $p$ consecutive steps $udd\cdots d u$. 
\end{theorem}

\begin{proof}
Let $\pi  \in \s_n(132,(p-1)(p-2)\cdots 2\hat 1p)$ and put $D_n=\theta(\pi)$. We prove that $D_n$ does not contain $ud\cdots du$. On the contrary, suppose that $D_n$ contains $ud\cdots du$ which are indexed $u_{\ell_{p}}d_{\ell_{p-1}}\cdots d_{\ell_2}u_{\ell_1}$. Then we have the followings:

\begin{itemize}
\item $\ell_p=\ell_{p-1}$ and $\ell_1=\ell_2+1$ (by $(ii)$, $(iii)$).
\item $$\ell_1>\ell_1-1=\ell_2>\ell_3\cdots>\ell_{p-2}>\ell_{p-1},$$ since by $(i)$, the order of the indexed up and down steps in $D_n$ must be $$\cdots u_{\ell_2}\cdots u_{\ell_3}\cdots u_{\ell_{p-2}}\cdots u_{\ell_{p-1}}d_{\ell_{p-1}}d_{\ell_{p-2}}\cdots d_{\ell_{3}}d_{\ell_{2}}u_{\ell_1}\cdots d_{\ell_{1}}\cdots$$ Hence,   
\item $$\pi=\cdots \ell_2\cdots \ell_3\cdots \ell_{p-1}\ell_1\cdots$$ 
with  $\ell_{p-1}$ attached to $\ell_1$, since by the construction of $\theta$,  $\pi$ is determined by getting the indices of the up steps in the indexed Dyck path. 
\end{itemize}
Hence, $red(\ell_2 \ell_3 \cdots \ell_{p-1}\ell_1) = (p-2)(p-3)\cdots 1(p-1)$. Furthermore, 

\begin{itemize}
\item there is not any element $k$ between $\ell_{i-1}$ and $\ell_i$ in $\pi$ such that $\ell_{i} < k < \ell_{i-1}$ for $i=3,4,\dots,p-1,$ otherwise $(u_k,d_k)$ is within $(u_{\ell_{i-1}}, d_{\ell_{i-1}})$ and contains $(u_{\ell_{i}}, d_{\ell_{i}})$ which contradicts the hypothesis that $d_{\ell_{i-1}}$ attaches to $d_{\ell_{i}}$. 
\item there is not any element between $\ell_{p-1}$ and $\ell_1$.
\item there is not any element $k$ before $\ell_2$ in $\pi$ such that $\ell_2<k<\ell_1$ since $\ell_2=\ell_1-1$.    
\end{itemize}
Therefore, the subsequence $\ell_2 \ell_3 \cdots \ell_{p-1}\ell_1$ of $\pi$ can not be expanded at any positions into the subsequence whose reduction is $(p-1)(p-2)\cdots 1p$. This is a contradiction.

Conversely, it is proved similarly by contradiction that if  $\pi$ contains a subsequence $\ell_2\ell_3\cdots\ell_{p-1}\ell_1$, whose reduction is $(p-2)(p-3)\cdots 1(p-1)$, which is not able to be expanded into the pattern $(p-1)(p-2)\cdots 1p$ at any positions, then $\theta(\pi)$ contains $u_{\ell_{p-1}}d_{\ell_{p-1}}d_{\ell_{p-2}}\cdots d_{\ell_{2}}u_{\ell_1}$. 
\end{proof}

\begin{figure}[h]
\centering
\includegraphics[width=8cm]{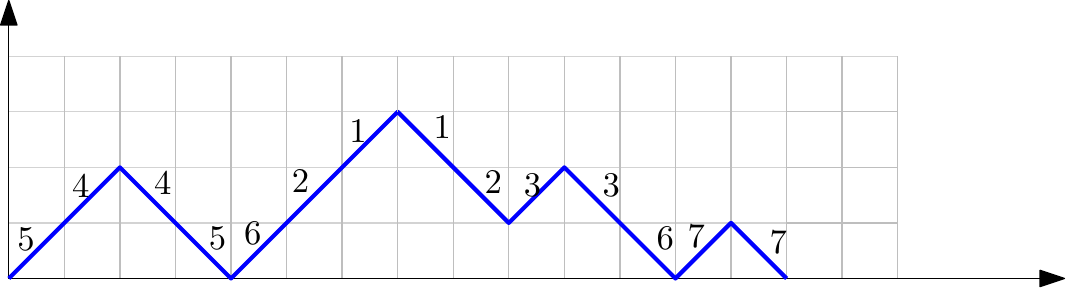}
\caption{The Dyck path with no $udu$ corresponding to $5462137\in \s_7(132, \hat 213)$ by $\theta$.}
\label{F:Dyck3}
\end{figure}
Figure \ref{F:Dyck3} illustrates a Dyck $7$-path with no $udu$ and its preimage by $\theta$ is $5462137$ which avoids both $132$ and $2\hat 13$.

The next Corollary \ref{C:Mot} is straightforward from Theorems \ref{T:udu} and Proposition \ref{P:cal}. 
\begin{proposition}[\cite{Cal04}]\label{P:cal} Dyck $n$-paths with no $udu$ counts $(n-1)$-th Motzkin number.
\end{proposition}
\begin{corollary}\label{C:Mot}
$\s_n(132,2\hat 13)$ counts $(n-1)$-th Motzkin numbers.
\end{corollary}
In Section \ref{S:4} we will give a new direct bijection from $\s_n(132,2\hat 13)$ to Motzkin $(n-1)$-paths. 
\subsection{Wilf-equivalence through Simion-Schmidt's bijection.} \label{S:33} In this section we use Simion-Schmidt's bijection to find some other patterns belonging to the same Wilf class to that investigated in Sections \ref{S:32}. 
$\ \\$

\begin{algorithm}
\KwIn{A permutation $\pi=\pi_1\pi_2\cdots\pi_n$ in $\s_n(132)$}
\KwOut{A permutation $\sigma=\sigma_1\sigma_2\cdots\sigma_n$ in $\s_n(123)$}

$\sigma_1:=\pi_1$ \;
$x:=\pi_1$ \;
\ForEach{$i = 2,\ldots, n$}{
	\eIf{$\pi_i < x$}{
		$\sigma_i:=\pi_i$ \;
		$x:=\pi_i$ \;}
	{
			$\sigma_i:=\max\{k|x<k\leq n, k\neq \sigma_j$ for all $j<i\}$\;
		}
	
}

\caption{Simion-Schmidt \cite{SS85}}\label{Algo_SS}
\end{algorithm}

The map $\pi \mapsto \sigma$ defined by Algorithm \ref{Algo_SS} is the Simion-Schmidt bijection \cite{SS85}. As an example, $7561234  \in \s_7(132)$ maps to $7561432 \in \s_7(123)$.

\begin{lemma}
Simion-Schmidt's map is a bijection from $\s_n(132,12\cdots\hat{p})$ to \\$\s_n(123,1\hat{p}(p-1)\cdots2)$.
\end{lemma}

\begin{proof}
According to Algorithm \ref{Algo_SS}, for any pattern $\tau=1k (k-1) \cdots 2$, $k>1$, in $\sigma$, lines $4, 5, 6$ and $8$ guarantee that

\begin{itemize}
	\item  the element in $\sigma$ corresponding to 1 in $\tau$ is equal to a pivot $x$. Thus, all other elements in $\sigma$ corresponding to $k,k-1,\cdots,2$ in $\tau$ are greater than $x$;
	\item  elements in $\pi$ corresponding to $k,k-1,\cdots,1$ in $\tau$ respectively are in the increasing order. 
\end{itemize}

Hence, each pattern $\tau_1=1p\cdots 32$ in $\sigma$ maps to the pattern $\tau_2=12\cdots p$ in $\pi$ and each pattern $\tau_1^*=1(p-1)\cdots 32$ in $\tau_1$ maps to the pattern $\tau_2^*=12\cdots (p-1)$ in $\tau_2$. Regarding the definition of the hatted pattern, if $\pi$ avoids $12\cdots\hat{p}$, it is, each pattern $\tau_2^*$ can be expanded to a pattern $\tau_2$ in $\pi$ at a   position, then $\sigma$ avoids the pattern $1\hat{p}(p-1)\cdots2$, it is, each pattern $\tau_1^*$ can be expanded to a pattern $\tau_1$ in $\sigma$ at a position. Thus if $\pi \in \s_n(132,12\cdots\hat{p})$ then $\sigma \in \s_n(123,1\hat{p}(p-1)\cdots2)$.
\end{proof}

Consequently, the permutations in $\s_{n+1}(123,1\hat{p}(p-1)\cdots2)$ are also viewed as  Dyck $n$-paths with no peak at height $p$. 
\section{Hatted pattern visiting Motzkin paths}\label{S:4}
In this section, we study the set of permutations $S_n(132,2\hat 13)$ which is a special case of the investigated class of permutations in Section \ref{S:31} when $p=3$ . We prove the set of permutations $\s_n(132,2\hat 13)$ is equal to the set of permutations $\s_n(132)$ without two adjacent consecutive numbers. Furthermore, we give a new explicit bijection from $\s_n(132,2\hat 13)$ to Motzkin $(n-1)$-paths as mentioned in the previous section. Finally, by using the ECO method, we show that the Motzkin generating tree coded by permutations in $\s_n(132,3\bar 142)$ \cite{} is now well coded by permutations in $\s_n(132,2\hat 13)$. 
\subsection{Direct bijection between $\s_{n+1}(132,2\hat 13)$ and Motzkin $n$-paths}\label{S:41} 

We first recall the definition of Motzkin $n$-paths on a horizontal line, unlike the usual definition of Motzkin paths which start and end on the $x$-axis. 
\begin{definition} Let $n,h$ be non-negative integers. 
\begin{itemize}
\item[(i)] A \emph{Motzkin $n$-path on the line $y=h$} is a lattice path in the integer plane starting and ending at points on $y=h$, which consists of $n$ steps including \emph{up} $(1,1)$, \emph{down} $(1,-1)$ and \emph{flat} $(1,0)$ ones and never runs bellow  $y=h$
\item[(ii)] A Motzkin $n$-path on the line $y=h$ is called \emph{proper on $y=h$} if there is not any flat-steps on $y=h$. 
\end{itemize}
\end{definition}

Motzkin paths are also represented by words of length $n$ on $\{u,d,f\}$ where $u$, $d$, $f$ substitute for up, down and flat step respectively. For the sake of expression of what following, we call two adjacent consecutive numbers $a(a+1)$ in a permutation is \emph{the factor} $a(a+1)$.
 
\begin{theorem} \label{T:fac}
Let $\pi\in\s_n$. Then $\pi\in\s_n(132, 2\hat 13)$ if and only if $\pi\in\s_n(132)$ and $\pi$ does not contain any factor $a(a+1)$ for $1 \leq a \leq n-1$. 
\end{theorem}
\begin{proof}
In order to prove the ``if" part, on the contrary we assume that $\pi$ contains a factor $a(a+1)$ for some $1 \leq a \leq n-1$. So $a(a+1)$ is also a subsequence of $\pi$ and has reduction $12$. This factor can not be expanded into the pattern $213$ in $\pi$ at any positions, before $a$ as well as after $(a+1)$. Therefore, $\pi$ contains $2\hat13$ which is a contradiction.

\noindent Conversely, it is sufficient to prove that $\pi$ avoids $2\hat13$. Let us take an increasing subsequence $\pi_i\pi_j$ of $\pi$ with $\pi_i<\pi_j$ and $i<j$. Since $\pi$ avoids $132$, its $LTRM$-blocks satisfy the conditions in Lemma \ref{L:ltrm}. We consider two following cases: 
\begin{itemize}
\item $\pi_i$ and $\pi_j$ are in two different $LTRM$-blocks. Then $\pi_j$ not a $LTR$ minimum. So $\pi_i$, the minimum of the $LTRM$- block containing $\pi_j$, and $\pi_j$ form a subsequence of $\pi$ whose reduction is $213$. 
\item $\pi_i$ and $\pi_j$ are in the same $LTRM$-block. Since $\pi$ does not contain the factor $\pi_i(\pi_i+1)$ and the elements in the same block form an increasing sequence, $\pi_j\geq \pi_i+2$ and $\pi_i+1$ is not between $\pi_i$ and $\pi_j$ in $\pi$. Furthermore, $\pi_i+1$ is not after $\pi_j$, otherwise $\pi$ contains the subsequence  $\pi_i\pi_j(\pi_i+1)$ whose reduction is $132$. In other words, $\pi_i+1$ must appear before $\pi_i$ in $\pi$. Therefore, $(\pi_i+1) \pi_i \pi_j$ is a subsequence of $\pi$ whose reduction is $213$.
\end{itemize}
 Hence, $\pi$ avoids $2\hat 13$ in any cases. 
\end{proof}
Now let $\pi\in \s_n$ with $LTRM$ indices $i_1,\dots,i_k$ such that its $LTRM$-blocks satisfy the conditions in Lemma \ref{L:ltrm}, that is  
\begin{itemize}
\item[i)] the first elements of the blocks are decreasing from left to right;
\item[ii)] each block is an increasing sequence. 
\end{itemize}

In this case, we also say that $\pi$ has a \emph{geometric representation} which is the union of all representations of $LTRM$-blocks of $\pi$. Each $LTRM$-block $\pi_{i_t}\dots\pi_{i_{t+1}-1}$, for $t=1,\dots, k-1$, of $\pi$ is represented by a set of $(i_{t+1}-i_t)$ semi-circles from $(\pi_i,0)$ to $(\pi_{i+1},0)$ on the upper half-plane for $i=i_t,\dots, i_{t+1}-2$. If a $LTRM$-block has only one element then its representation is a single point. See Figure \ref{F:repre_per} as an example. Two $LTRM$-blocks of $\pi$ are called \emph{overlapping} if there exist two intersecting semi-circles in their representations. Figure \ref{F:crossover} illustrates the overlapping and non-overlapping properties of  two different $LTRM$-blocks.   
\begin{figure}[h]
\centering
\includegraphics[width=12cm]{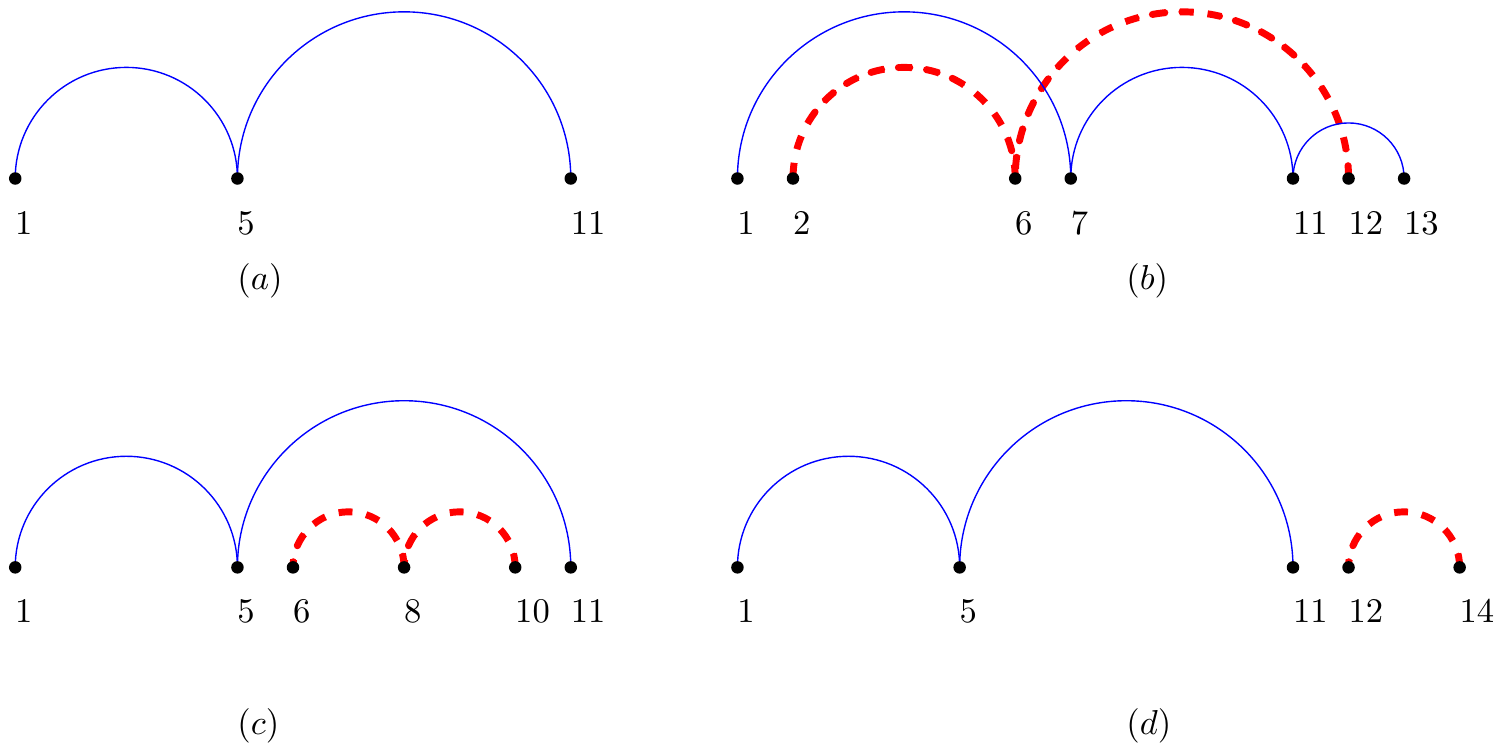}
\caption{(a): The representation of $LTRM$-block $(1, 5, 11)$; (b): Two $LTRM$-blocks $(1, 7, 11, 13)$ and $(2, 6, 12)$ are overlap; (c),(d): Two $LTRM$-blocks are non-overlap.}
\label{F:crossover}
\end{figure}

\begin{lemma}\label{L:132} Let $\pi\in\s_n$ such that $\pi$ has a geometric representation. Then $\pi\in \s_n(132)$ if and only if its $LTRM$-blocks are pairwise non-overlapping.
\end{lemma}
\begin{proof}
On the contrary, we assume that there exists two semi-circles from $(\pi_i,0)$ to$(\pi_{i+1},0)$ with $\pi_i<\pi_{i+1}$ and from $(\pi_j,0)$ to $(\pi_{j+1},0)$ with $\pi_j<\pi_{j+1}$ intersecting. Without loss of generality, we can assume that $\pi_i<\pi_j$. We have 
$$\pi_i<\pi_j<\pi_{i+1}<\pi_{j+1}.$$ 

So $\pi$ contains either $\pi_i\pi_{i+1}\pi_j$ (if $i<j$) or $\pi_j\pi_{j+1}\pi_i$ (if $j<i$) as its own subsequence. Moreover, $red(\pi_i\pi_{i+1}\pi_j)=red(\pi_j\pi_{j+1}\pi_i)=132$. This contradicts the $132$-avoiding property of $\pi$.

Conversely, assume that $\pi$ contains a subsequence $\pi_i\pi_j\pi_k$ such that $red(\pi_i\pi_j\pi_k)=132$. It is remarkable that $\pi_j$ and $\pi_k$ are neither $LTR$ minima nor in the same $LTRM$-block. The $LTR$ minimum of the $LTRM$-block containing $\pi_j$ is greater than that of the $LTRM$-block containing $\pi_k$. Therefore, the $LTRM$-blocks containing $\pi_j$ and $\pi_k$ are overlapping which is a contradiction. 
\end{proof}
\begin{figure}[h]
\subfigure[]{\label{F:repre_per}\includegraphics[width=5cm]{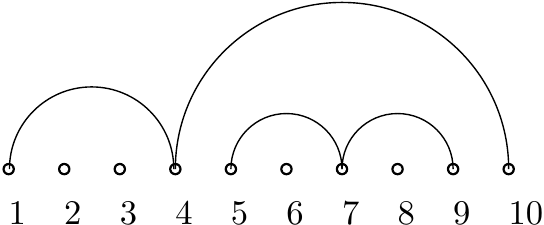}}\hspace{1cm}
\subfigure[]{\label{F:re_per1}\includegraphics[width=6cm]{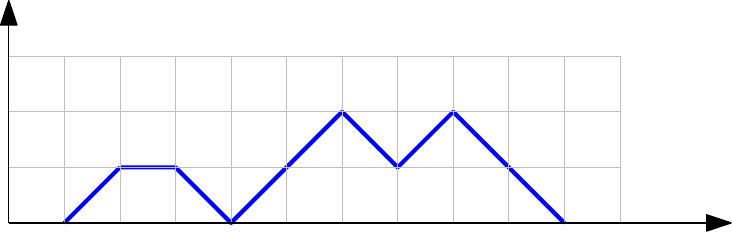}}
\caption{$(a)$: Geometric representation of $(8)(6)(5\ 7\ 9)(3)(2)(1\ 4\ 10)$; $(b)$: The Motzkin path $\psi(8\ 6\ 5\ 7\ 9 \ 3 \ 2\ 1 \ 4 \ 10)$}\label{F:re_per}
\end{figure}
We now present a new direct bijection, denoted by $\psi$, from Motzkin $n$-paths to $\s_{n+1}(132,2 \hat 13)$. Let $M_n$ be a Motzkin $n$-path from $(1,0)$ to $(n+1,0)$. The image $\psi(M_n)$ is a permutation on $[n+1]$ determined by the $LTRM$-blocks as follows: 
\begin{itemize}
\item[(i)] For each non-negative integer $h$ no greater than the height of $M_n$, we consider all proper Motzkin paths of maximal length on $y=h$ which are parts of $M_n$;
\item[(ii)] For each $h$, each such Motzkin path intersects to $y=h$ at some points whose abscissas rearranged in the increasing order create a $LTRM$-block of $\pi$;
\item[(iii)] Rearranging all blocks created in (ii) in the decreasing order of the first elements gives a permutation which is $\psi(M_n)$. 
\end{itemize}

\noindent For the example of the Motzkin $9$-path in Figure \ref{F:re_per1}, taking the intersection to $y=0$ we get the $LTRM$-block $(1\ 4\ 10)$; to $y=1$ we get three $LTRM$-blocks: $(2)$, $(3)$ and $(5\ 7\ 9)$; to $y=2$ we get two $LTRM$-blocks: $(8)$ and $(6)$. So its image by $\psi$ is $(8)(6)(5\ 7\ 9)(3)(2)(1\ 4\ 10)$. 

Conversely, let $\pi\in\s_{n+1}(132,2\hat 13)$. The inverse of $\pi$ by $\psi$ is a Motzkin $n$-path from $(1,0)$ to $(n+1,0)$ determined as follows. On the geometric representation of $\pi$, we consider turn by turn its points from $1$ to $n+1$.  When the position $i$ ($1\leq i \leq n+1$) is considered, we implement the following steps
\begin{enumerate}
\item  go-up one step if $i$ is a starting point of a semi circle;
\item go-flat one step if 
\begin{itemize}
\item $i$ is a single point and $i+1$ is either a single point or a starting point of a semi circle and simultaneously not an ending point of other semi-circle;
\item $i$ is an ending point of a semi circle and simultaneously not a starting point of other semi-circle and $i+1$ is either a single point or a starting point of a semi-circle;
\end{itemize}
\item go-down one step if 
\begin{itemize}
\item $i$ is a single point and $(i+1)$ is an ending point of a semi circular;
\item $i$ is an ending point of a semi circular and simultaneously not a starting point of other semi-circular and $i+1$ is an ending point of other semi circle;
\end{itemize}
\item do nothing if $i=n+1$.
\end{enumerate}  
\begin{theorem} \label{T:mot}
The map $\psi$ defined above is a bijection from Motzkin $n$-paths to $\s_{n+1}(132, 2\hat 13)$.
\end{theorem}
\begin{proof}
Let $M_n$ be a Motzkin $n$-path. By the construction of $\psi$, $\psi(M_n)$ is a permutation and its $LTRM$-blocks satisfy the conditions in Lemma \ref{L:ltrm}. We prove $\psi(M_n)$ is $132$-avoiding by showing its $LTRM$-blocks pairwise non-overlapping (by Lemma \ref{L:132}). Taking two different $LTRM$-blocks of $\psi(M_n)$, we consider two following cases:
\begin{itemize}
\item[i)] These two $LTRM$-blocks are created by taking the intersection of $M_n$ to the same line $y=h$ in the construction of $\psi$. Then two corresponding proper Motzkin paths on this line are separated by at least a flat-step or a path under this line. Hence, their representations are non-overlapping. 
\item[ii)] These two $LTRM$-blocks are created by taking the intersection of $M_n$ to two different lines, say $y=h_1$ and $y=h_2$  ($h_1<h_2$). Assume that two corresponding proper Motzkin paths on these lines are $M^1$ and $M^2$ respectively. Take $M^*$ the Motzkin path of $M_n$ which is of the smallest length satisfying
\begin{itemize}
\item $M^*$ contains $M^2$;
\item $M^*$ is a Motzkin path on the line $y=h_1$. 
\begin{figure}[h]
\centering
\includegraphics[height=4cm]{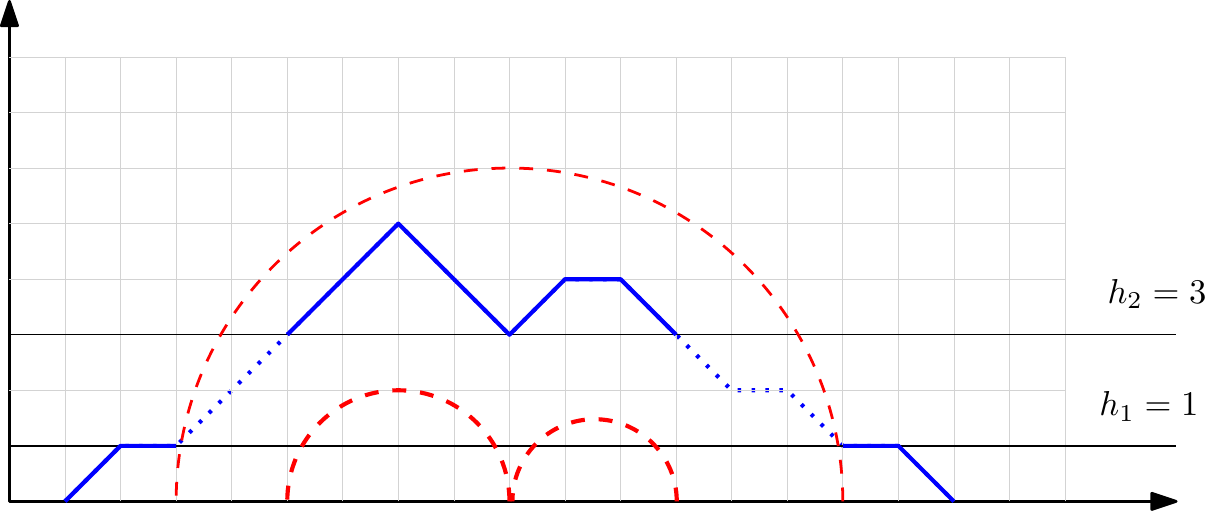}
\caption{Proper Motzkin on  $y=3$ and $y=1$}\label{F:Motz1}
\end{figure}
\end{itemize}
See Figure \ref{F:Motz1} for an illustration of $M^*$. Since $M^*$ is of smallest length and $h_1<h_2$, it is a proper Motzkin path on the line $y=h_1$ and moreover, it meets the line $y=h_1$ at exactly two points whose projection on $x$-axis are ending points of a semi-circle in the geometric representation of $\psi(M_n)$ (by the construction). Furthermore, the representation of the $LTRM$-block corresponding $M^2$ is within this semi-circle since $M^*$ contains $M^2$. By (i) this semi-circle is either a part of or disjoint to the  geometric representation of the $LTRM$-block corresponding to $M^1$. Hence, in any cases, two $LTRM$-blocks corresponding $M^1$ and $M^2$ are non-overlapping. 
\end{itemize}

On the other hand, by the construction of $\psi$, $a$ and $a+1$ cannot belong to the same $LTRM$-block since otherwise the Motzkin path corresponding to this $LTRM$-block must contain a flat-step connecting $a$ and $(a+1)$ which is not a proper Motzkin path on that line. Hence, $\psi(M_n)$ avoids all factors $a(a+1)$ for $a=1,2,\dots,n$. By Theorem \ref{T:fac}, $\psi(M_n)\in \s_{n+1}(132,2 \hat 13)$. 

Conversely, it is remarkable that in the inversion we constructed above, we always go down or go flat (depending on the role of each element $i$ and $(i+1)$ in the permutation) to separate the proper Motzkin paths on same levels. So it must be the inversion of $\psi$. This completes the proof. 
\end{proof}

\subsection{Motzkin generating tree}\label{S:42}

In this section, we apply the ECO method to show that the Motzkin generating tree given in \cite{BDPP00} coded by permutations in $\s_n(321,3\bar{1}42)$ is also coded by permutations without any factor $a(a+1)$ in $\s_{n+1}(132)$. See Figure \ref{fig:Motzkin-tree} for first levels of the tree.  

\begin{figure}[htb!]
   \centering
\includegraphics[width=12cm]{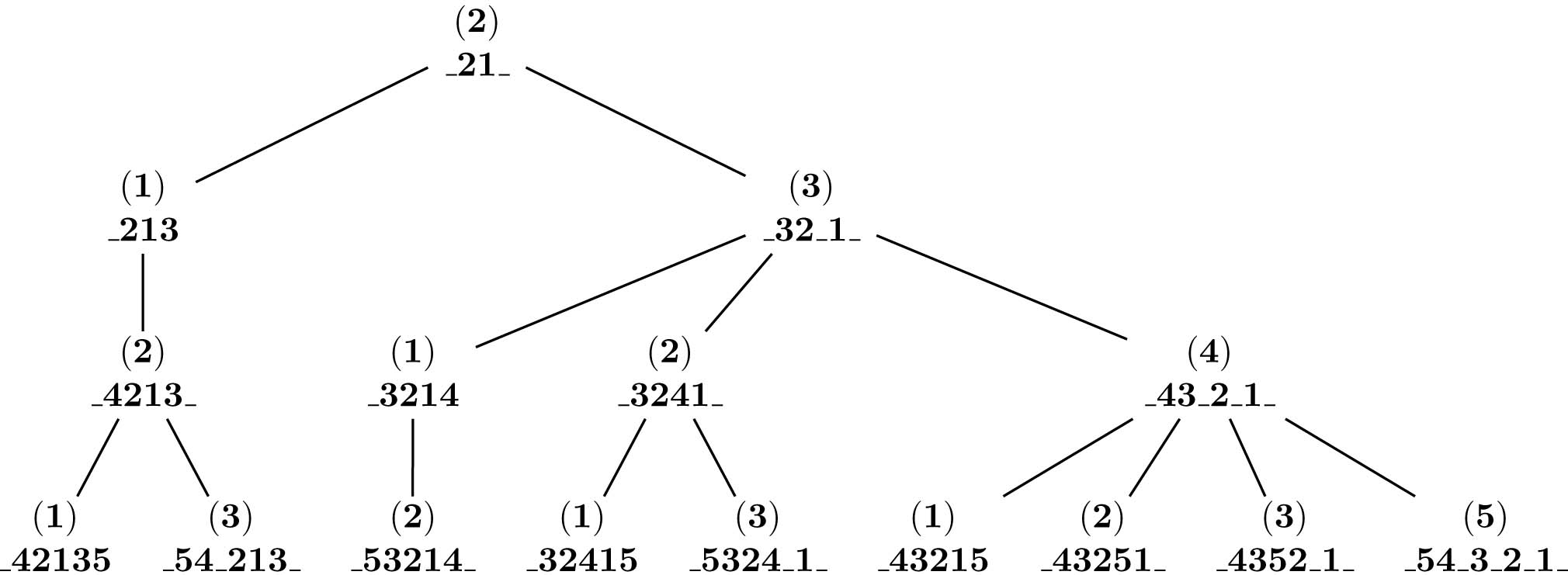}
    \caption{\small The first levels of the Motzkin generating tree coded by permutations avoiding $132$ and without any two consecutive integers adjacent. Each symbol `$\_$' represents an active site.}
    \label{fig:Motzkin-tree}
\end{figure}
The ECO method is used for the
enumeration and the recursive construction of combinatorial object
classes. This is a recursive description explaining how an object of size $n$ can be reached from one and only one object of inferior size \cite{BDPP98,BDPP00}.
More precisely, it consists to give a system of {\it succession
rules} for a combinatorial object class which induces a {\it generating
tree} such that each node is labeled: the set of successions rules
describes for each node the label of its successors. Generating trees are usually coded by permutations. The root is often coded by the identity of length one. Let $\pi$ be an
$n$-length permutation in a generating tree; each successor $\pi$ is
obtained from $\pi$ by inserting $n+1$ into certain positions also known as the {\it active
sites} of $\pi$. Notice that the sites are numbered from left to
right, from $1$ to $n+1$. Denote by $\pi^{\downarrow i}$ the permutation obtained from $\pi$ by inserting $n+1$ to its $i^{th}$ site.

In \cite{BDPP00}, a system of succession rules of Motzkin generating tree $(\Omega)$  is given by  :

\[(\Omega) \left\{\begin{array}{ll}
            (2) \\
            (k) \rightsquigarrow (k+1) (k-1) (k-2) \ldots (2) (1) .
        \end{array}
   \right. \]
   
 \begin{theorem} 
   Each level $n$ of the generating tree $(\Omega)$ can be coded by the permutations without any factor $a(a+1)$ in $\s_{n+1}(132)$. The root is coded by $21$.
\end{theorem}

\begin{proof} Let $\pi=\pi_1\pi_2\ldots\pi_n$ be any node in $(\Omega)$ having label $k$. Suppose that the indexes of its actives site are numbered from the left  $s_1,s_2, \ldots, s_k$. We have two following remarks:
\begin{enumerate}  
	\item If inserting $n+1$ creates a $132$, then  $n+1$ plays the role of $3$;
	\item If inserting $n+1$ creates a factor $a(a+1)$, then $n+1$ plays the role of $a+1$ and $n$ plays the role of $a$.
\end{enumerate}
	These two remarks cannot happen if $n+1$ becomes the leftmost element, so the site $1$ is always active, $s_1=1$. Furthermore, considering $n+2$ sites in $\pi^{\downarrow 1}$,
\begin{itemize}  
	\item the first site of $\pi^{\downarrow 1}$ is active;
	\item $k-1$ sites $s_2+1,s_3+1,\ldots, s_k+1$ in $\pi^{\downarrow 1}$ are active;
	\item the site just after $n$, which is not active in $\pi$ (otherwise it contains $n(n+1)$), is active in $\pi^{\downarrow 1}$ because: 
	\begin{itemize} 
		\item $n(n+2)$ is not of type $a(a+1)$,
		\item the subsequences $(n+1)(n+2)\pi_j$ and $n(n+2)\pi_j$ are not of pattern $132$, and
		\item any subsequence $\pi_i(n+2)\pi_j$ in $\pi^{\downarrow 1}$ is not of pattern 132, otherwise the subsequence $\pi_in\pi_j$ in $\pi$ is of pattern $132$. 
	\end{itemize}
	\item Since $(n+1)$ cannot play the role as $1$ and also $2$ in a pattern $132$ of $\pi^{\downarrow 1}$, the role in creating a pattern $132$ of $n+1$ in $\pi$ is the same to the one of $n+2$ in $\pi^{\downarrow 1}$. Hence, if $s$ is a site inactive in $\pi$, then $s+1$ is also inactive in $\pi^{\downarrow 1}$. 
\end{itemize}

Therefore, $\pi^{\downarrow 1}$ has label $k+1$.

On the other hand, suppose $n+1$ is inserted into active site $s_u$ ($u \geq 2$). Then the first site and the sites  $s_{u+1}+1,\dots,s_{k}+1$ of $\pi^{\downarrow s_u}$ are active. Furthermore, all sites except for the first site to the left $s_u+1$ of $\pi^{\downarrow s_u}$ are inactive since otherwise $\pi^{\downarrow s_u}$ contains the subsequence $\pi_1(n+2)(n+1)$ which is of pattern $132$. Moreover, $n$ is always on the left of $n+1$ in $\pi^{\downarrow s_u}$, otherwise $\pi^{\downarrow s_u}$ contains the subsequence $\pi_1(n+1)n$ which is of pattern $132$. Similarly as above, all inactive sites in $\pi$ shifted by $1$ are also not inactive in $\pi^{\downarrow s_u}$. Hence, there are $k-u+1$ sites of $\pi^{\downarrow s_u}$ active.

The successors of $\pi$ in $(\Omega)$ thus receive the labels $k+1,k-1,\ldots,2,1$, respectively, as the active sites are considered in order from left to right.

\end{proof}


\bibliographystyle{plain}
\bibliography{ECO,Hat}

\end{document}